\documentclass[11pt,reqno]{amsart}
\usepackage{fullpage,verbatim,paralist}
\usepackage[breaklinks,pdfstartview=FitH]{hyperref}
\usepackage{algorithmic}
\usepackage{algorithm}

\usepackage{amssymb,fullpage,mathrsfs,verbatim,graphicx,paralist}

\newcommand{\eqdef}{\stackrel{\mathrm{def}}{=}}
\renewcommand{\le}{\leqslant}

\renewcommand{\ge}{\geqslant}

\renewcommand{\setminus}{\smallsetminus}

  \DeclareMathOperator{\diam}{diam}

\newcommand{\N}{\mathbb{N}}
\newcommand{\E}{\mathbb{E}}

\newcommand{\U}{\mathscr{U}}

\renewcommand{\P}{\mathscr P}
\newcommand{\e}{\varepsilon}
\renewcommand{\S}{\mathcal{S}}
\theoremstyle{plain}
  \newtheorem{lemma}{Lemma}[section]

  \newtheorem{theorem}[lemma]{Theorem}
  
  \newtheorem{corollary}[lemma]{Corollary}

  \theoremstyle{definition}

  \newtheorem{remark}[lemma]{Remark}

\begin{document}

\title{Ultrametric skeletons}
\thanks{M. M. was partially supported by ISF grants 221/07 and  93/11,
BSF grants 2006009 and 2010021, and a gift from Cisco Research Center. A. N.
was partially supported by NSF grant CCF-0832795, BSF grants
2006009 and 2010021, and the Packard Foundation. Part of this work was completed when M. M. was a visiting researcher at Microsoft Research and A. N. was visiting the Quantitative Geometry program at the Mathematical Sciences Research Institute.}
\author{Manor Mendel}
\address{Mathematics and Computer Science Department, Open University of Israel, 1 University Road, P.O. Box 808
Raanana 43107, Israel} \email{mendelma@gmail.com}

\author{Assaf Naor}
\address{Courant Institute, New York University, 251 Mercer Street, New York NY 10012, USA}
\email{naor@cims.nyu.edu}

\begin{abstract} We prove that for every $\e\in (0,1)$ there exists $C_\e\in (0,\infty)$ with the following property. If $(X,d)$ is a compact metric space and $\mu$ is a Borel probability measure on $X$ then there exists a compact subset $S\subseteq X$ that embeds into an ultrametric space with distortion $O(1/\e)$, and a probability measure $\nu$ supported on $S$ satisfying $\nu\left(B_d(x,r)\right)\le \left(\mu(B_d(x,C_\e r)\right)^{1-\e}$ for all $x\in X$ and $r\in (0,\infty)$. The dependence of the distortion on $\e$ is sharp. We discuss an extension of this statement to multiple measures, as well as how it implies Talagrand's majorizing measures theorem.
\end{abstract}

\maketitle


\section{Introduction}

Our main result is the following theorem.
\begin{theorem}\label{thm:our measure}
For every $\e\in (0,1)$ there exists $C_\e\in (0,\infty)$ with the following property.  Let $(X,d)$ be a compact metric space and let $\mu$ be a Borel probability measure on $X$. Then there exists a compact subset $S\subseteq X$  satisfying
\begin{enumerate}
\item $S$ embeds into an utrametric space with distortion $O(1/\e)$.
\item There exists a Borel probability measure supported on $S$ satisfying
\begin{equation}\label{eq:measure growth}
\nu\left(B_d(x,r)\right)\le \left(\mu(B_d(x,C_\e r)\right)^{1-\e}
\end{equation}
for all $x\in X$ and $r\in [0,\infty)$.
\end{enumerate}
\end{theorem}
Recall that an ultrametric space is a metric space $(U,\rho)$ satisfying the strengthened triangle inequality $\rho(x,y)\le \max\{\rho(x,z),\rho(y,z)\}$ for all $x,y,z\in U$. Saying that $(S,d)$ embeds with distortion $D\in [1,\infty)$ into an ultrametric space means that there exists an ultrametric space $(U,\rho)$ and an injection $f:S\to U$ satisfying $d(x,y)\le \rho(f(x),f(y))\le Dd(x,y)$ for all $x,y\in S$. In the statement of Theorem~\ref{thm:sub meas}, and in the rest of this paper, given a metric space $(X,d)$, a point $x\in X$ and a radius $r\in [0,\infty)$, the corresponding closed ball is denoted $B_d(x,r)=\{y\in X:\ d(y,x)\le r\}$, and the corresponding open ball is denoted $B_d^\circ(x,r)=\{y\in X:\ d(y,x)< r\}$. (We explicitly indicate the underlying metric since the ensuing discussion involves multiple metrics on the same set.)

We call the metric measure space $(S,d,\nu)$ from Theorem~\ref{thm:our measure} an {\em ultrametric skeleton} of the metric measure space $(X,d,\mu)$. The literature contains several theorems about the existence of ``large" ultrametric subsets of metric spaces; some of these results will be mentioned below. As we shall see, the subset $S$ of Theorem~\ref{thm:our measure} must indeed be large, but it is also geometrically  ``spread out" with respect to the initial probability measure $\mu$. For example, if $\mu$ assigns positive mass to two balls $B_d(x,r)$ and $B_d(y,r)$, where $x,y\in X$ satisfy $d(x,y)> C_\e( r+1)$, then the probability measure $\nu$, which is supported on $S$, cannot assign full mass to any one of these balls. This is one reason why $(S,d,\nu)$ serves as a ``skeleton" of $(X,d,\mu)$.

More significantly, we call $(S,d,\mu)$ an ultrametric skeleton because it can be used to deduce global information about the entire initial metric measure space $(X,d,\mu)$. Examples of such global applications of statements that are implied by Theorem~\ref{thm:our measure} are described in~\cite{MN07,MN11-ultra}, and an additional example will be presented below. As a qualitative illustration of this phenomenon, consider a stochastic process $\{Z_t\}_{t\in T}$, assuming for simplicity that the index set $T$ is finite and that each random variable $Z_t$ has finite second moment. Equip $T$ with the metric $d(s,t)=\sqrt{\E\left[(Z_s-Z_t)^2\right]}$. Assume that there exists a unique (random) point $\tau\in T$ satisfying $Z_\tau=\max_{t\in T} Z_t$. Let $\mu$ be the law of $\tau$, and apply Theorem~\ref{thm:our measure}, say, with $\e=1/2$, to the metric measure space $(X,d,\mu)$. One obtains a subset $S\subseteq X$ that embeds into an ultrametric space with distortion $O(1)$, and a probability measure $\nu$ that is supported on $S$ and satisfies~\eqref{eq:measure growth} (with $\e=1/2$). If $\sigma\in S$ is a random point of $S$ whose law is $\nu$, then it follows that for every $x\in T$, $p\in (0,1)$ and $r\in [0,\infty)$, if $\sigma$ falls into $B_d(x,r)$ with probability at least $p$, then the global maximum $\tau$ falls into $B_d(x,O(r))$ with probability at least $p^2$. One can therefore always find a subset of $T$ that is more structured due to the fact that it is approximately an ultrametric space (e.g., such structure can be harnessed for chaining-type arguments), yet this subset reflects the location of the global maximum of $\{Z_t\}_{t\in T}$ in the above distributional/geometric sense. A quantitatively sharp variant of the above qualitative interpretation of Theorem~\ref{thm:our measure} is discussed in Section~\ref{sec:majorizing} below.

\subsection{Nonlinear Dvoretzky theorems}\label{sec:dvo} Nonlinear Dvoretzky theory, as initiated by Bourgain, Figiel and Milman~\cite{BFM86}, asks for theorems asserting that any ``large" metric space contains a ``large" subset that embeds with specified distortion into Hilbert space. We will see below examples of notions of ``largeness" of a metric space for which a nonlinear Dvoretzky theorem can be proved.  For an explanation of the relation of such problems to the classical Dvoretzky theorem~\cite{Dvo60}, see~\cite{BFM86,BLMN05,MN11-ultra}. Most known nonlinear Dvoretzky theorems actually obtain subsets that admit a low distortion embedding into an ultrametric space. Since ultrametric spaces admit an isometric embedding into Hilbert space~\cite{VT79}, such a result falls into the Bourgain-Figiel-Milman framework.  Often (see~\cite{BFM86,BLMN05}) one can prove an asymptotically matching impossibility result which shows that all subsets of a given metric space that admit a low distortion embedding into Hilbert space must be ``small". Thus, in essence, it is often the case that the best way to find an almost Hilbertian subset is actually to aim for a subset satisfying the seemingly more stringent requirement of being almost ultrametric.

Apply Theorem~\ref{thm:our measure} to an $n$-point metric space $(X,d)$, with $\mu(\{x\})=1/n$ for all $x\in X$. Since  $\nu$ is a probability measure on $S$, there exists $x\in X$ with $\nu(\{x\})\ge 1/|S|$. An application of~\eqref{eq:measure growth} with $r=0$ shows that $1/|S|\le \mu(\{x\})^{1-\e}=1/n^{1-\e}$, or $|S|\ge n^{1-\e}$. Since $S$ embeds into an ultrametric space with distortion $O(1/\e)$, this shows that Theorem~\ref{thm:our measure} implies the sharp solution of the Bourgain-Figiel-Milman nonlinear Dvoreztky problem that was first obtained in~\cite{MN07}. Sharpness in this context means that, as shown in~\cite{BLMN05}, there exists a universal constant $c\in (0,\infty)$ and for every $n\in \N$ there exists an $n$-point metric space $X_n$ such that every $S\subseteq X_n$ with $|S|\ge n^{1-\e}$ incurs distortion at least $c/\e$ in any embedding into Hilbert space. Thus, the distortion bound in Theorem~\ref{thm:our measure} cannot be improved (up to constants), even if we allow $S$ to embed into Hilbert space.

Assume that $(X,d)$ is a compact metric space of Hausdorff dimension greater than $\alpha\in (0,\infty)$. Then there exists~\cite{How95,Mattila} an $\alpha$-Frostman measure on $(X,d)$, i.e., a Borel probability measure $\mu$ satisfying $\mu(B_d(x,r))\le Kr^\alpha$ for every $x\in X$ and $r\in (0,\infty)$, where $K$ is a constants that may depend  on $X$ and $\alpha$ but not on $x$ and $r$. An application of Theorem~\ref{thm:our measure} to $(X,d,\mu)$ yields a compact subset $S\subseteq X$ that embeds into an ultrametric space with distortion $O(1/\e)$, and a Borel probability measure $\nu$ supported on $S$ satisfying $\nu(B_d(x,r))\le \mu(B_d(x,C_\e r))^{1-\e}\le K^{1-\e}C_\e^{(1-\e)\alpha}r^{(1-\e)\alpha}$ for all $x\in X$ and $r\in (0,\infty)$. Hence $\nu$ is a $(1-\e)\alpha$-Frostman measure on $S$, implying~\cite{Mattila} that $S$ has Hausdorff dimension at least $(1-\e)\alpha$. Thus Theorem~\ref{thm:our measure} implies the sharp solution of Tao's nonlinear Dvoretzky problem for Hausdorff dimension that was first obtained in~\cite{MN11-ultra}.

More generally, the following result was proved in~\cite{MN11-ultra} as the main step towards the solution of Tao's nonlinear Dvoretzky problem for Hausdorff dimension.
\begin{theorem}\label{thm:sub meas}
For every $\e\in (0,1)$ there exists $c_\e= e^{O(1/\e^2)}\in (0,\infty)$ with the following property.  Let $(X,d)$ be a compact metric space and let $\mu$ be a Borel probability measure on $X$. Then there exists a compact subset $S\subseteq X$  satisfying
\begin{enumerate}
\item $S$ embeds into an utrametric space with distortion $O(1/\e)$.
\item If $\{x_i\}_{i\in I}\subseteq X$ and $\{r_i\}_{i\in I}\subseteq [0,\infty)$ satisfy $\bigcup_{\in I}B_d(x_i,r_i)\supseteq S$ then
\begin{equation}\label{eq:ball cover}
\sum_{i\in I} \mu\left(B_d(x_i,c_\e r_i)\right)^{1-\e}\ge 1.
\end{equation}
\end{enumerate}
\end{theorem}
Theorem~\ref{thm:sub meas} is a consequence of Theorem~\ref{thm:our measure}. Indeed, if $S\subseteq X$ and $\nu$ are the subset and probability measure from Theorem~\ref{thm:our measure}, then  $1=\nu(S)\le \sum_{i\in I}\nu(B_d(x_i,r_i))\le \sum_{i\in I}\mu(B_d(x_i,C_\e r_i))^{1-\e}$ whenever $\bigcup_{i\in I} B_d(x_i,r_i)\supseteq S$. But, Theorem~\ref{thm:sub meas} is the main reason for the validity of the phenomenon described in Theorem~\ref{thm:our measure}: here we show how to formally deduce Theorem~\ref{thm:our measure} from Theorem~\ref{thm:sub meas}, with $C_\e=O(c_\e/\e)=e^{O(1/\e^2)}$. Alternatively, with more work, one can repeat the proof of Theorem~\ref{thm:sub meas} in~\cite{MN11-ultra} while making changes to several lemmas in order to prove Theorem~\ref{thm:our measure} directly, and obtain $C_\e=c_\e$. Since the proof of Theorem~\ref{thm:sub meas} in~\cite{MN11-ultra} is quite involved, we believe that it is instructive to establish Theorem~\ref{thm:our measure} via the argument described here.


 \subsubsection{Majorizing measures and stochastic processes}\label{sec:majorizing} Theorem~\ref{thm:our measure} makes it possible to relate the nonlinear Dvoretzky framework of~\cite{MN11-ultra} to Talagrand's nonlinear Dvoretzky theorem~\cite{Tal87}, and consequently to Talagrand's majorizing measures theorem~\cite{Tal87}. Given a metric space $(X,d)$ let $\P_X$ be the Borel probability measures on $X$. The Fernique-Talagrand $\gamma_2$ functional is defined as follows.
\begin{equation}\label{eq:def gamma2}
\gamma_2(X,d)=\inf_{\mu\in \P_X}\sup_{x\in X} \int_0^\infty \sqrt{\log\left(\frac{1}{\mu(B(x,r))}\right)}dr.
\end{equation}
Talagrand's nonlinear Dvoretzky theorem~\cite{Tal87} asserts that every finite metric space $(X,d)$ has a subset $S\subseteq X$ that embeds into an ultrametric space with distortion $O(1)$ and\footnote{Here, and in what follows, the relations $\lesssim,\gtrsim$ indicate the corresponding inequalities up to factors which are universal constants. The relation $A \asymp B$ stands for $(A\lesssim B) \wedge (A \gtrsim B)$.} $\gamma_2(S,d)\gtrsim\gamma_2(X,d)$. Talagrand proved this nonlinear Dvoretzky theorem in order to prove his celebrated majorizing measures theorem, which asserts that if $\{G_x\}_{x\in X}$ is a Gaussian process and for $x,y\in X$ we set $d(x,y)=\sqrt{\E\left[(G_x-G_y)^2\right]}$, then $\E\left[\sup_{x\in X} G_x\right]\gtrsim \gamma_2(X,d)$. There is also a simpler earlier matching upper bound due to Fernique~\cite{Fer76}, so $\E\left[\sup_{x\in X} G_x\right]\asymp\gamma_2(X,d)$. The fact that Talagrand's nonlinear Dvoretzky theorem implies the majorizing measures theorem is simple; see~\cite[Prop.~13]{Tal87} and also the discussion in~\cite[Sec.~1.3]{MN11-ultra}.

To understand the link between Theorem~\ref{thm:our measure} and Talagrand's nonlinear Dvoretzky theorem, consider the following quantity, associated to every compact metric space $(X,d)$.
\begin{equation}\label{eq:def delta2}
\delta_2(X,d)=\sup_{\mu\in \P_X}\inf_{x\in X} \int_0^\infty \sqrt{\log\left(\frac{1}{\mu(B(x,r))}\right)}dr.
\end{equation}
Intuitively, $\gamma_2(X,d)$ should be viewed as a multi-scale version of a covering number, while $\delta_2(X,d)$ should be viewed  as a multi-scale version of a packing number.  It is therefore not surprising that $\gamma_2(X,d)\asymp \delta_2(X,d)$. In fact, in Section~\ref{sec:gamma} we note that $\gamma_2(U,\rho)=\delta_2(U,\rho)$ for every finite ultrametric space $(U,\rho)$, and $\delta_2(X,d)\ge \gamma_2(X,d)$ for every finite metric space $(X,d)$
(the latter inequality is an improvement of our original bound $\delta_2(X,d)\gtrsim \gamma_2(X,d)$, due to an elegant argument of Witold Bednorz~\cite{Bed11}).
The remaining estimate $\delta_2(X,d)\lesssim \gamma_2(X,d)$ will not be needed here, though it follows from our discussion (see Remark~\ref{re:formal}), and it also has a simpler direct proof.

Let $\mu\in \P_X$ satisfy
$
\delta_2(X,d)=\inf_{x\in X} \int_0^\infty \sqrt{\log\left(1/\mu(B(x,r))\right)}dr.
$
Theorem~\ref{thm:our measure} applied to $(X,d,\mu)$ yields  $S\subseteq X$ and an ultrametric $\rho:S\times S\to [0,\infty)$ satisfying $d(x,y)\le \rho(x,y)\le Kd(x,y)$ for all $x,y\in S$. Additionally, there exists  $\nu\in \P_S$ satisfying $\nu\left(B_d(x,r)\right)\le \sqrt{\mu(B_d(x,Kr))}$ for all $x\in X$ and $r\in [0,\infty)$. Here $K\in (0,\infty)$ is a universal constant. Since $B_\rho(x,r)\subseteq B_d(x,r)\cap S\subseteq B_\rho(x,Kr)$ for all $x\in S$ and $r\in [0,\infty)$, we have $\delta_2(S,d)\le \delta_2(S,\rho)=\gamma_2(S,\rho)\le K\gamma_2(S,d)$, where we used the fact that $\gamma_2(\cdot)$ and $\delta_2(\cdot)$ coincide for ultrametrics. Hence,




\begin{multline}\label{eq:deduce tal}
K\gamma_2(S,d)\ge \delta_2(S,d)\ge \inf_{x\in S} \int_0^\infty \sqrt{\log\left(\frac{1}{\nu(B_d(x,r))}\right)}dr\ge \inf_{x\in S} \int_0^\infty \sqrt{\log\left(\frac{1}{\sqrt{\mu(B_d(x,Kr))}}\right)}dr\\= \frac{1}{K\sqrt{2}} \inf_{x\in S} \int_0^\infty \sqrt{\log\left(\frac{1}{\mu(B_d(x,r))}\right)}dr\ge \frac{\delta_2(X,d)}{K\sqrt{2}}\ge \frac{\gamma_2(X,d)}{K\sqrt{2}}.
\end{multline}
This completes the deduction of Talagrand's nonlinear Dvoretzky theorem from Theorem~\ref{thm:our measure}.

\begin{remark}\label{re:formal}
It is easy to check (see~\cite[Lem.~6]{Tal87}) that $\gamma_2(S,d)\le 2\gamma_2(X,d)$ for every $S\subseteq X$  (and, even more trivially, $\delta_2(S,d)\le \delta_2(X,d)$). Thus, it follows from~\eqref{eq:deduce tal} that $\gamma_2(X,d)\gtrsim \delta_2(X,d)$ for every finite metric space $(X,d)$.
\end{remark}

Since the original 1987 publication of Talagrand's majorizing measures theorem, this theorem has been reproved and simplified in several subsequent works, yielding important applications and generalizations (mainly due to Talagrand himself). These proofs are variants of the same basic idea:  a greedy top-down construction, in which one looks at a given scale for a ball on which a certain functional is maximized, removes a neighborhood of this ball, and iterates this step on the remainder of the metric space. It seems that the framework described here is genuinely different. The proof of Theorem~\ref{thm:sub meas} in~\cite{MN11-ultra} has two phases. One first constructs a nested family of partitions in a bottom-up fashion: starting with singletons one iteratively groups the points together based on a gluing rule that is tailor-made in anticipation of the ensuing ``pruning" or ``sparsification" step. This second step is a top-down iterative removal of appropriately ``sparse" regions of the partitions that were constructed in the first step; here one combines  an analytic argument with the pigeonhole principle to show that there are sufficiently many potential pruning locations so that successive iterations of this step can be made to align appropriately. Our new approach  has the advantage that it yields distributional statements such as~\eqref{eq:measure growth}, the majorizing measures theorem itself being a result of integrating these pointwise estimates.

\subsection{Multiple measures}\label{sec:union intro} In anticipation of further applications of ultrametric skeletons, we end by addressing what is perhaps the simplest question that one might ask about these geometric objects: to what extent is the union of two ultrametric skeleton also an ultrametric skeleton? 

We show in Remark~\ref{rem:D1D2 sharp} that for arbitrarily large $D_1,D_2\in [1,\infty)$, one can find a finite metric space $(X,d)$, and two disjoint subsets $U_1,U_2\subseteq X$, such that each $U_i$ embeds into an ultrametric space with distortion $D_i$, yet any embedding of $U_1\cup U_2$ into an ultrametric space incurs distortion at least $(D_1+1)(D_2+1)-1$. In Section~\ref{sec:union ultra} we prove the following geometric result of independent interest (which, as explained above, is sharp up to lower order terms).


\begin{theorem}\label{thm:UMunion}
Fix $D_1,D_2\in [1,\infty)$. Let $(X,d)$ be a metric space and
$U_1,U_2\subseteq X$. Assume that $(U_1,d)$ embeds with distortion
$D_1$ into an ultrametric space and that $(U_2,d)$ embeds with distortion
$D_2$ into an ultrametric space. Then the metric
space $(U_1\cup U_2,d)$ embeds with distortion at most
$(D_1+2)(D_2+2)-2$ into an ultrametric space.
\end{theorem}

Consequently, one can always find an ultrametric skeleton that is ``large" with respect to any finite list of probability measures.

\begin{corollary}\label{coro:many measures}
For every $\e\in (0,1)$ let $C_\e$ be as in Theorem~\ref{thm:our measure}. Let $(X,d)$ be a compact metric space, and let $\mu_1,\ldots,\mu_k$ be Borel probability measures on $X$. Then there exists a compact subset $S\subseteq X$, and Borel probability measures $\nu_1,\ldots,\nu_k$ supported on $S$, such that $S$ emebds into an ultrametric space with distortion at most $(O(1)/\e)^k$ and for every $x\in X$ and $r\in [0,\infty)$ we have $\nu_i\left(B_d(x,r)\right)\le \left(\mu_i(B_d(x,C_\e r)\right)^{1-\e}$ for all $i\in \{1,\ldots,k\}$.
\end{corollary}

\section{Proof of Theorem~\ref{thm:our measure}}\label{sec:proof}

A {\em submeasure} on a set $X$ is a function $\xi:2^X\to [0,\infty)$ satisfying the following conditions.
\begin{enumerate}
\item[(a)] $\xi(\emptyset) =0$,
\item[(b)] $A_1\subseteq A_2\subseteq X\implies \xi(A_1)\le \xi(A_2)$,
\item[(c)] $\{A_i\}_{i\in I}\subseteq X\implies \xi\left(\bigcup_{i\in I}A_i\right)\le \sum_{i\in I}\xi(A_i)$.
\end{enumerate}
If in addition $\xi(X)=1$ we call $\xi$ a {\em probability submeasure}.

\begin{lemma}\label{lem:sub-nosub}
Let $(U,\rho)$ be a compact ultrametric space, and let $\xi:2^U\to [0,\infty)$ be a probability submeasure. Then there exists a Borel probability measure $\nu$ on $U$ satisfying $\nu(B_\rho(x,r))\le \xi(B_\rho(x,r))$ for all $x\in X$ and $r\in [0,\infty)$.
\end{lemma}

\begin{remark}\label{rem:pathological}
It is known~\cite{HC75,Tal80} that there exist probability submeasures that do not dominate any nonzero measure (in the literature such measures are called pathological submeasures). Lemma~\ref{lem:sub-nosub} shows that probability submeasures on ultrametric spaces always dominate {\em on all balls} some probability measure.

\end{remark}

Assuming the validity of Lemma~\ref{lem:sub-nosub} for the moment, we prove Theorem~\ref{thm:our measure}.
\begin{proof}[Proof of Theorem~\ref{thm:our measure}]
Let $(X,d)$ be a compact metric space and $\mu$ a Borel probability measure on $X$. By Theorem~\ref{thm:sub meas} there exists a compact subset $S\subseteq X$ satisfying the covering estimate~\eqref{eq:ball cover}, and an ultrametric $\rho:S\times S\to [0,\infty)$ satisfying $d(x,y)\le \rho(x,y)\le \frac{K}{\e}d(x,y)$ for all $x,y\in S$, where $K$ is a universal constant.

For every $A\subseteq S$ define
\begin{equation}\label{eq:ded xi}
\xi(A)=\inf\left\{\sum_{i\in I} \mu\left(B_d(x_i,c_\e r_i)\right)^{1-\e}:\ \{(x_i,r_i)\}_{i\in I}\subseteq X\times [0,\infty)\ \wedge\  \bigcup_{i\in I} B_d(x_i,r_i)\supseteq A\right\}.
\end{equation}
In~\eqref{eq:ded xi} the index set $I$ can be countably infinite or finite, with the convention that an empty sum vanishes. One checks that $\xi:2^S\to [0,\infty)$ is a { submeasure} on $S$. Moreover, for every $x\in X$ and $r\in [0,\infty)$, by considering $B_d(x,r)$ as covering itself, we deduce from~\eqref{eq:ded xi} that
\begin{equation}\label{eq:xi small}
\xi\left(S\cap B_d(x,r)\right)\le \mu\left(B_d(x,c_\e r)\right)^{1-\e}.
\end{equation}
Since $\mu$ is a probability measure and $X$ has bounded diameter, it follows from~\eqref{eq:xi small} that $\xi(S)\le 1$. The covering estimate~\eqref{eq:ball cover} implies that $\xi(S)\ge 1$, so in fact $\xi$ is a probability submeasure on $S$.

An application of Lemma~\ref{lem:sub-nosub} to $(S,\rho,\xi)$ yields a Borel probability measure $\nu$ supported on $S$ and satisfying $\nu(B_\rho(y,r))\le \xi(B_\rho(y,r))$ for all $y\in S$ and $r\in [0,\infty)$. Fix $x\in X$ and $r\in [0,\infty)$. The desired estimate~\eqref{eq:measure growth} holds trivially if $B_d(x,r)\cap S=\emptyset$, so we may assume that there exists $y\in S$ with $d(x,y)\le r$. Thus
$$
S\cap B_d(x,r)\subseteq S\cap B_d(y,2r)\subseteq B_\rho\left(y,\frac{2K}{\e}r\right)\subseteq S\cap B_d\left(y,\frac{2K}{\e}r\right)\subseteq S\cap B_d\left(x,\left(1+\frac{2K}{\e}\right)r\right).
$$
It follows that
\begin{multline*}
\nu(B_d(x,r))\le \nu\left(B_\rho\left(y,\frac{2K}{\e}r\right)\right)\le \xi\left(B_\rho\left(y,\frac{2K}{\e}r\right)\right)\\\le \xi\left(S\cap B_d\left(x,\left(1+\frac{2K}{\e}\right)r\right)\right)\le \mu \left(B_d\left(x,c_\e\left(1+\frac{2K}{\e}\right)r\right)\right)^{1-\e}.
\end{multline*}
This completes the deduction of Theorem~\ref{thm:our measure} from Theorem~\ref{thm:sub meas} and Lemma~\ref{lem:sub-nosub}.
\end{proof}

Prior to proving Lemma~\ref{lem:sub-nosub}, we review some basic facts about compact ultrametric spaces; see~\cite{Hug04} for an extended and more general treatment of this topic. Fix a compact ultrametric space $(U,\rho)$. For every $r\in (0,\infty)$ we have $|\{B_\rho(x,s):\ (x,s)\in U\times [r,\infty)\}|<\infty$, i.e.,  there are only finitely many closed balls in $U$ of radius at least $r$. Indeed, by compactness  $U$ contains only finitely many disjoint closed balls of radius at least $r$. Since  $B_\rho(x,s)\cap B_\rho(y,t)\in \{\emptyset,B_\rho(x,s),B_\rho(y,t)\}$ for every $x,y\in U$ and $s,t\in [0,\infty)$, assuming for contradiction that $\{B_\rho(x,s):\ (x,s)\in U\times [r,\infty)\}$ is infinite, we deduce that there exist $\{(x_i,s_i)\}_{i=1}^\infty\subseteq U\times [r,\infty)$ satisfying $B_\rho(x_i,s_i)\subsetneq B_\rho(x_{i+1},s_{i+1})$ for all $i\in \N$. Fix $y_i\in B_\rho(x_{i+1},s_{i+1})\setminus B_\rho(x_i,s_i)$. If $i<j$ then $y_j\notin B_\rho(x_{j},s_j)\supseteq B_\rho(x_{i+1},s_{i+1})$ and $y_i\in B_\rho(x_{i+1},s_{i+1})$. Hence $ s_{i+1}< \rho(y_j,x_{i+1})\le \max\{\rho(y_j,y_i),\rho(y_i,x_{i+1})\}\le \max\{\rho(y_j,y_i),s_{i+1}\}$. It follows that $\rho(y_i,y_j)> s_{i+1}\ge r$ for all $j>i$, contradicting the compactness of $(U,\rho)$.

A consequence of the above discussion is that for every $x\in U$ and $r\in (0,\infty)$ there exists $\e\in (0,\infty)$ such that $B_\rho(x,r)=B_\rho(x+\e)$. Therefore $B_\rho(x,r)=B_\rho^\circ(x,r+\e/2)$. Similarly, since $B_\rho^\circ (x,r)=\bigcup_{\delta\in (0,r/2]} B_\rho(x,r-\delta)$, where there are only finitely many distinct balls appearing this union, there exists $\delta\in (0,r/2]$ such that $B_\rho^\circ(x,r)=B_\rho(x,r-\delta)$. Thus every open ball in $U$ of positive radius is also a closed ball, and every closed ball in $U$ of positive radius is also and open ball. Consider the equivalence relation on $U$ given by $x\sim y\iff \rho(x,y)<\diam_\rho(U)$. This is indeed an equivalence relation since $\rho$ is an ultrametric. The corresponding equivalence classes are all of the form $B_\rho^\circ(x,\diam_\rho(U))$ for some $x\in U$. Being open sets that cover $U$, there are only finitely many such equivalence classes, say, $\{B_1^1,B^2,\ldots,B_1^{k_1}\}$. By the above discussion, each of the open balls $B_1^i$ is also a closed ball, and hence $(B_1^i,\rho)$ is a compact ultrametric space.  We can therefore continue the above construction iteratively, obtaining a sequence  $\{P_j\}_{j=0}^\infty$ of partitions of $U$ with the following properties.
\begin{enumerate}
\item $P_0=\{U\}$.
\item $P_j$ is finite for all $j$.
\item $P_{j+1}$ is a refinement of $P_j$ for all $j$.
\item Every $C\in P_j$ is of the form $B_\rho^\circ(x,r)$ for some $x\in U$ and $r\in [0,\infty)$.
\item For every $j$, if $C\in P_j$ is not  a singleton then  there exists $x_1,\ldots,x_k\in U$ such that $\{B_\rho^\circ(x_i,\diam_\rho(C))\}_{i=1}^k\subseteq P_{j+1}$, the open balls $\{B_\rho^\circ(x_i,\diam_\rho(C))\}_{i=1}^k$ are disjoint, and $C=\bigcup_{i=1}^k B_\rho^\circ(x_i,\diam_\rho(C))$.
\item $\lim_{j\to \infty} \max_{C\in P_j} \diam_\rho(C)=0$.
\item For every $x\in U$ and $r\in (0,\infty)$ there exists $j$ such that $B_\rho^\circ(x,r)\in P_j$.
\end{enumerate}
The first five items above are valid by construction. The sixth item follows from the fact that for all $j\in \N$ either $P_{j-1}$ consists of singletons or  $\max_{C\in P_j} \diam_\rho(C)<\max_{C\in P_{j-1}} \diam_\rho(C)$. Since for every $r\in (0,\infty)$ there are only finitely many balls of radius at least $r$ in $U$, necessarily $\lim_{j\to \infty} \max_{C\in P_j} \diam_\rho(C)=0$. To prove the seventh item above, assume for contradiction that $(x,r)\in U\times (0,\infty)$ is such that $B_\rho^\circ(x,r)\notin P_j$ for all $j$. Since the set $\{B_\rho^\circ(x,s)\}_{s\ge r}$ is finite, and $B_\rho^\circ(x,\diam(U)+1)=U\in P_0$, we may assume without loss of generality that $B_\rho^\circ(x,s)\in \bigcup_{j=0}^\infty P_j$ for all $s\in (r,\infty)$. But since $B_\rho(x,r)=B^\circ_\rho(x,s)$ for some $s\in (r,\infty)$, it follows that $B_\rho(x,r)\in P_j$ for some $j$. In particular, $B_\rho^\circ(x,r)\neq B_\rho(x,r)$, implying that $\diam_\rho(B_\rho(x,r))=r$. Therefore by construction $B_\rho^\circ(x,r)\in P_{j+1}$, a contradiction.

\begin{proof}[Proof of Lemma~\ref{lem:sub-nosub}]
Let $\{P_j\}_{j=0}^\infty$ be the sequence of partitions of $U$ that was constructed above. We will first define $\nu$ on $\S=\bigcup_{j=0}^\infty P_j\cup\{\emptyset\}$, which is the set of all open balls in $U$ (allowing the radius to vanish, in which case the corresponding open ball is empty). Setting $\nu(X)=1$ and $\nu(\emptyset)=0$, assume inductively that $\nu$ has been defined on $P_j$. For $C\in P_{j+1}$ let $D\in P_j$ be the unique set satisfying $C\subseteq D$. There exist disjoint sets $C_1,\ldots,C_k\in P_{j+1}$, with $C\in \{C_1,\ldots,C_k\}$, such that $D=C_1\cup\cdots\cup C_k$.  Define
\begin{equation}\label{eq:def nu}
\nu(C)=\frac{\xi(C)}{\sum_{i=1}^k\xi(C_i)}\cdot \nu(D).
\end{equation}
This completes the inductive definition of $\nu:\S\to [0,\infty)$.

We claim that one can apply the Carath\'eodory extension theorem to extend $\nu$ to a Borel measure on $U$. To this end, note that $\S$ is a semi-ring of sets. Indeed, $\S$ is closed under intersection since $C\cap D\in \{\emptyset, C,D\}$ for all $C,D\in \S$. We therefore need to check that for every $C,D\in \S$, the set $D\setminus C$ is a finite disjoint union of elements in $\S$. For this purpose we may assume that $D\setminus C\neq \emptyset$, implying that $C\subsetneq D$. Assume that $C\in P_j$ and $D\in P_i$ for $i<j$. Let $C_1,\ldots,C_k\in P_j$ be the distinct elements of $P_j$ that are contained in $D$, enumerated so that $C=C_1$. Then $D\setminus C=C_2\cup\cdots\cup C_k$, and this union is disjoint, as required.

In order to apply the Carath\'eodory extension theorem, it remains to check that if $\{A_i\}_{i=1}^\infty\subseteq \S$ are pairwise disjoint and $\bigcup_{i=1}^\infty A_i\in \S$, then $\nu\left(\bigcup_{i=1}^\infty A_i\right)=\sum_{i=1}^\infty \nu(A_i)$. Since all the elements of $\S$ are both open and closed, compactness implies that it suffices to show that if $A_1,\ldots,A_m\in \S$ are pairwise disjoint and $A_1\cup\cdots\cup A_m\in \S$, then $\nu(A_1\cup\cdots\cup A_m)=\sum_{i=1}^m\nu(A_i)$. We proceed by induction on $m$, the case $m=1$ being vacuous. For every $i\in \{1,\ldots, m\}$ there is a unique $k_i\in \N$ such that $A_i\in P_{k_i-1}\setminus P_{k_i}$. Define $k=\max\{k_1,\ldots,k_m\}$. If $k=1$ then necessarily $m=1$ and $A_1=U$. Assume that $k>1$ and fix $j\in \{1,\ldots,m\}$ satisfying $k_j=k$. Let $D\in P_{k-2}$ be the unique element of $P_{k-2}$ containing $A_j$, and let $C_1,\ldots,C_\ell\in P_{k-1}$ be the distinct elements of $P_{k-1}$ contained in $D$. Since $A_1\cup\cdots \cup A_m$ is a ball containing $A_j\subseteq D$, we have $A_1\cup\cdots \cup A_m\supseteq D=C_1\cup\cdots\cup C_\ell$. By maximality of $k$ it follows that $A_j\in \{C_1,\ldots,C_\ell\}\subseteq \{A_1,\ldots,A_m\}$. For $i\in \{1,\ldots,\ell\}$ let $n_i\in \{1,\ldots,m\}$ be such that $C_i=A_{n_i}$. Since $\left(\bigcup_{i\in \{1,\ldots,m\}\setminus \{n_1,\ldots,n_\ell\}}A_i\right)\bigcup D=\bigcup_{i=1}^m A_i$, the inductive hypothesis implies that $\sum_{i\in \{1,\ldots,m\}\setminus \{n_1,\ldots,n_\ell\}}\nu(A_i)+\nu(D)=\nu(A_1\cup\cdots\cup A_m)$. But by our definition~\eqref{eq:def nu} we have $\nu(D)=\nu(A_{n_1})+\cdots+\nu(A_{n_\ell})$, so that indeed $\nu(A_1\cup\cdots\cup A_m)=\sum_{i=1}^m\nu(A_i)$.

Having defined the Borel probability measure $\nu$, it remains to check by induction on $j$ that if $C\in P_j$ then $\nu(C)\le \xi(C)$. If $j=0$ then $C=U$ and $\nu(U)=\xi(U)=1$. If $j\ge 1$ then let $D\in P_{j-1}$ satisfy $C\subseteq D$. There exist disjoint sets $C_1,\ldots,C_k\in P_{j+1}$, with $C\in \{C_1,\ldots,C_k\}$, such that $D=C_1\cup\cdots\cup C_k$. Since $\xi$ is a submeasure, $\xi(D)\le \xi(C_1)+\cdots+\xi(C_k)$. By the inductive hypothesis $\nu(D)\le \xi(D)$. Our definition~\eqref{eq:def nu} now implies that $\nu(C)\le \xi(C)$. The proof of Lemma~\ref{lem:sub-nosub} is complete.
\end{proof}

\section{$\gamma_2(X,d)$ and $\delta_2(X,d)$}\label{sec:gamma}

Let $(X,d)$ be a finite metric space. For every measurable $\phi:(0,\infty)\to [0,\infty)$ define
$$
\gamma_\phi(X,d)=\inf_{\mu\in \P_X} \sup_{x\in X} \int_0^\infty \phi(\mu(B_d(x,r)))dr,
$$
and
$$
\delta_\phi(X,d)=\sup_{\mu\in \P_X} \inf_{x\in X} \int_0^\infty \phi(\mu(B_d(x,r)))dr.
$$
Thus $\gamma_2(\cdot)=\gamma_\phi(\cdot)$ and $\delta_2(\cdot)=\delta_\phi(\cdot)$ for $\phi(x)=\sqrt{\log(1/x)}$. The following lemma is a variant of an argument of Bednorz~\cite[Lem.~4]{Bed11}. The elegant proof below was shown to us by Keith Ball; it is a generalization and a major simplification of our original proof of the estimate $\delta_2(X,d)\gtrsim \gamma_2(X,d)$.
\begin{lemma}\label{lem:bed-ball}
Assume that $\phi:(0,\infty)\to (0,\infty)$ is continuous and $\lim_{x\to 0^+}\phi(x)=\infty$.  Then $\delta_\phi(X,d)\ge \gamma_\phi(X,d)$.
\end{lemma}

\begin{proof}
Write $X=\{x_1,\ldots,x_n\}$. Thus $\P_X$ can be identified with the $(n-1)$-dimensional simplex $\Delta_{n-1}=\{(\mu_1,\ldots,\mu_n)\in [0,1]^n;\ \mu_1+\cdots+\mu_n=1\}$ (by setting $\mu\{x_i\})=\mu_i$).

Define $f_1,\ldots,f_n:\Delta_{n-1}\to [0,\infty)$ by
$$
f_i(\mu)=\left\{\begin{array}{ll}0&\mathrm{if\ }\mu_i=0,\\
\left(1+\int_0^\infty \phi(\mu(B(x_i,r)))dr\right)^{-1}&\mathrm{if\ }\mu_i>0.\end{array}\right.
$$
Writing $S(\mu)=\sum_{i=1}^n f_i(\mu)$, we define $F:\Delta_{n-1}\to \Delta_{n-1}$ by $F(\mu)=(f_1(\mu),\ldots,f_n(\mu))/S(\mu)$. Since $\phi$ is continuous and $\lim_{x\to 0^+}\phi(x)=\infty$, all the $f_i$ are continuous on $\Delta_{n-1}$. Since each $\mu\in \Delta_{n-1}$ has at least one positive coordinate, $S(\mu)>0$. Thus $F$ is continuous. Note that by definition $F$ maps each face of $\Delta_{n-1}$ into itself. By a standard reformulation of the Brouwer fixed point theorem (see, e.g.,  \cite[Sec.~$4.29\frac12_+$]{Gro07}), it follows that $f(\Delta_{n-1})=\Delta_{n-1}$. In particular, there exists $\mu\in \Delta_{n-1}$ for which $F(\mu)=(1/n,\ldots,1/n)$. In other words, there exists $\mu\in \P_X$ such that $\int_{0}^\infty  \phi(\mu(B_d(x,r)))dr$ does not depend on $x\in X$. Hence,
\begin{equation*}
\delta_\phi(X,d)\ge  \inf_{x\in X} \int_0^\infty \phi(\mu(B_d(x,r)))dr= \sup_{x\in X} \int_0^\infty \phi(\mu(B_d(x,r)))dr\ge \gamma_\phi(X,d).\qedhere
\end{equation*}
\end{proof}

\begin{lemma}\label{lem:ultra equals}
Assume that $\phi:(0,\infty)\to [0,\infty)$ is non-increasing. Let $(U,\rho)$ be a finite ultrametric space. Then $\delta_\phi(U,\rho)\le \gamma_\phi(U,\rho)$.
\end{lemma}

\begin{proof}
We claim that if $\mu, \nu$ are nonnegative measures on $U$ satisfying $\mu(U)\le \nu(U)$ then there exists $a\in U$ satisfying $\mu(B_\rho(a,r))\le \nu(B_\rho(a,r))$ for all $r\in (0,\infty)$. This would imply the desired estimate since if $\mu,\nu\in \P_X$ are chosen so that  $\sup_{x\in X} \int_0^\infty \phi(\mu(B_\rho(x,r)))dr=\gamma_\phi(U,\rho)$ and $\inf_{x\in X} \int_0^\infty \phi(\mu(B_\rho(x,r)))dr=\delta_2(U,\rho)$, then
$$
\gamma_\phi(U,\rho)\ge \int_0^\infty \phi(\mu(B_\rho(a,r)))dr\ge \int_0^\infty \phi(\nu(B_\rho(a,r)))dr\ge \delta_2(U,\rho).
$$

The proof of the existence of $a\in U$ is by induction on $|U|$. If $|U|=1$ there is nothing to prove. Otherwise, as explained in Section~\ref{sec:proof}, there exist $x_1,\ldots,x_k\in U$ such that the balls $\{B_\rho^\circ(x_i,\diam_\rho(U))\}_{i=1}^k$ are nonempty, pairwise disjoint, and $\bigcup_{i=1}^k B_\rho^\circ(x_i,\diam_\rho(U))=U$. It follows that $\sum_{i=1}^k \mu(B_\rho^\circ(x_i,\diam_\rho(U)))=\mu(U)\le\nu(U)=\sum_{i=1}^k \nu(B_\rho^\circ(x_i,\diam_\rho(U)))$. Consequently there exists $i\in \{1,\ldots,k\}$ such that $\mu(B_\rho^\circ(x_i,\diam_\rho(U)))\le \nu(B_\rho^\circ(x_i,\diam_\rho(U)))$.  By the inductive hypothesis there exists $a\in B_\rho^\circ(x_i,\diam_\rho(U))$ satisfying $\mu(B_\rho(a,r))\le \nu(B_\rho(a,r))$ for all $r<\diam_\rho(U)$. Since for $r\ge \diam_\rho(U)$ we have $B_\rho(a,r)=U$, the proof is complete.
\end{proof}

A combination of Lemma~\ref{lem:bed-ball} and Lemma~\ref{lem:ultra equals} yields the following corollary.

\begin{corollary}\label{cor:equals on ultra}
If $\phi:(0,\infty)\to [0,\infty)$ is non-increasing, continuous, and $\lim_{x\to 0^+}\phi(x)=\infty$,  then $\delta_\phi(U,\rho)= \gamma_\phi(U,\rho)$ for all finite ultrametric spaces $(U,\rho)$.
\end{corollary}

\begin{remark}\label{rem:star} Consider the star metric $d_n$ on $\{0,1,\ldots,n\}$, i.e., $d_n(0,i)=1$ for all $i\in \{1,\ldots,n\}$ and $d_n(p,q)=2$ for all distinct $p,q\in \{1,\ldots,n\}$. The measure $\nu$ on $\{0,1,\ldots,n\}$ given by $\nu(\{0\})=0$ and $\nu(\{i\})=1/n$ for $i\in \{1,\ldots,n\}$, shows that $\delta_2(\{0,1,\ldots,n\},d_n)\ge 2\sqrt{\log n}$. At the same time, the measure $\mu$ on $\{0,1,\ldots,n\}$ given by $\mu(\{0\})=1/2$ and $\mu(\{i\})=1/(2n)$ for $i\in \{1,\ldots,n\}$, shows that $\gamma_2(\{0,1,\ldots,n\},d_n)\le \sqrt{\log(2n)}+\sqrt{\log(2n/(n+1))}\le (1/2+o(1))\delta_2(\{0,1,\ldots,n\},d_n)$. Thus, unlike the case of ultrametric spaces, for general metric spaces it is not always true that $\gamma_2(X,d)=\delta_2(X,d)$. Of course, due to Lemma~\ref{lem:bed-ball} and Remark~\ref{re:formal} we know that $\gamma_2(X,d)\asymp \delta_2(X,d)$. It seems plausible that always $\delta_2(X,d)\le 2 \gamma_2(X,d)$, but we do not investigate this here.
\end{remark}

\section{Unions of approximate ultrametrics}\label{sec:union ultra}

In this section we prove Theorem~\ref{thm:UMunion} and present some related examples. Below, given a partition $P$ of a set $X$, for $x\in X$ we denote by $P(x)$ the element of $P$ to which $x$ belongs.

\begin{lemma}\label{lem:union frag}
Fix  $D_1,D_2\ge 1$. Let $(X,d)$ be a metric space and let
$U_1,U_2\subseteq X$ be two bounded subsets of $X$. Assume that
$(U_1,d)$ embeds with distortion $D_1$ into an ultrametric space and that
$(U_2,d)$ embeds with distortion $D_2$ into an ultrametric space. Then for
every $\e\in (0,1)$ there is a partition $P$ of $U_1\cup U_2$ with
the following properties.
\begin{itemize}
\item For every $C\in P$,
\begin{equation}\label{eq:diam bound union}
\diam_d(C)\le (1-\delta)\diam_d(U_1\cup U_2),
\end{equation}
where
\begin{equation}\label{eq:def delta}
\delta\eqdef\frac{2\e
D_2}{(D_1D_2+2D_1+2D_2+2)(D_1D_2+2D_1+2D_2+2+\e)}.
\end{equation}
\item For every distinct $C_1,C_2\in P$,
\begin{equation}\label{eq:sep union}
d(C_1,C_2)\ge \frac{\diam_d(U_1\cup U_2)}{D_1D_2+2D_1+2D_2+2+\e}.
\end{equation}
\end{itemize}
\end{lemma}

\begin{proof} By rescaling we may assume that $\diam_d(U_1\cup U_2)=1$. Let
$\rho_1$ be an ultrametric on $U_1$ satisfying $d(x,y)\le
\rho_1(x,y)\le D_1d(x,y)$ for all $x,y\in U_1$. Define
\begin{equation}\label{eq:def a}
a\eqdef\frac{D_1D_2+2D_1}{D_1D_2+2D_1+2D_2+2},
\end{equation}
and consider the equivalence relation on $U_1$ given by $x\sim_1 y
\iff \rho_1(x,y)\le a$ (this is an equivalence relation since
$\rho_1$ is an ultrametric). Let $\{E_i\}_{i\in I}\subseteq 2^{U_1}$ be
the corresponding equivalence classes. Thus
\begin{equation}\label{eq:diam E_i}
\diam_d(E_i)\le \diam_{\rho_1}(E_i)\le a \end{equation}
for all
$i\in I$, and for distinct $i,j\in I$ we
have
\begin{equation}\label{eq:sep E_i}
d(E_i,E_j)\ge \frac{\rho_1(E_i,E_j)}{D_1}\ge \frac{a}{D_1}.
\end{equation}

Let $\rho_2$ be an ultrametric on $U_2$ satisfying $d(x,y)\le
\rho_2(x,y)\le D_2d(x,y)$ for all $x,y\in U_2$. Define
\begin{equation}\label{eq:def b}
b\eqdef\frac{D_2}{D_1D_2+2D_1+2D_2+2+\e},
\end{equation}
and consider similarly the equivalence relation on $U_2$ given by
$x\sim_2 y \iff \rho_2(x,y)\le b$. The corresponding equivalence
classes will be denoted $\{F_j\}_{j\in J}\subseteq 2^{U_2}$. Thus
\begin{equation}\label{eq:diam F_i}
\diam_d(F_i)\le \diam_{\rho_1}(F_i)\le b
\end{equation}
for all
$i\in J$, and for distinct $i,j\in J$ we
have
\begin{equation}\label{eq:sep F_i}
d(F_i,F_j)\ge \frac{\rho_2(F_i,F_j)}{D_2}\ge \frac{b}{D_2}.
\end{equation}

For every $i\in I$ denote
\begin{equation}\label{eq:def J_i}
J_i\eqdef\left\{j\in J:\ d(E_i,F_j)\le c\right\},
\end{equation} where
\begin{equation}\label{eq:def c}
c\eqdef\frac{1}{D_1D_2+2D_1+2D_2+2}.
\end{equation} Note that for every $j\in J$
there is at most one $i\in I$ for which $j\in J_i$.
Indeed, if $j\in J_i\cap J_\ell$, where $i\neq \ell$, then
\begin{eqnarray*}
d(E_i,E_\ell)&\le&
d(E_i,F_j)+\diam_d(F_j)+d(F_j,E_\ell)\\&\stackrel{\eqref{eq:diam
F_i}\wedge\eqref{eq:def J_i}}{\le}& 2c+b\\&\stackrel{\eqref{eq:def
b}\wedge\eqref{eq:def
c}}{=}&\frac{2}{D_1D_2+2D_1+2D_2+2}+\frac{D_2}{D_1D_2+2D_1+2D_2+2+\e}\\&<&
\frac{2+D_2}{D_1D_2+2D_1+2D_2+2}\stackrel{\eqref{eq:def a}}{=}
\frac{a}{D_1},
\end{eqnarray*}
Contradicting~\eqref{eq:sep E_i}.

Consider the partition $P$ of $U_1\cup U_2$ consisting of
the sets
$$
\left\{E_i\cup\left(\bigcup_{j\in J_i}
F_j\right)\right\}_{i\in I}\quad \mathrm{and}\quad \{F_j\setminus
U_1\}_{j\in J\setminus \left(\bigcup_{i\in I}
J_i\right)}.
$$
It follows from~\eqref{eq:sep E_i}, \eqref{eq:sep F_i}
and~\eqref{eq:def J_i} that for every distinct $C_1,C_2\in P$,
$$
d(C_1,C_2)\ge
\min\left\{\frac{a}{D_1},\frac{b}{D_2},c\right\}=\frac{1}{D_1D_2+2D_1+2D_2+2+\e}.
$$
Thus the partition $P$ satisfies~\eqref{eq:sep union}. It also
follows from~\eqref{eq:diam E_i}, \eqref{eq:diam F_i}
and~\eqref{eq:def J_i} that for every $C\in P$,
$$
\diam_d(C)\le a+2b+2c=(1-\delta),
$$
where we used the definitions~\eqref{eq:def delta}, \eqref{eq:def
a}, \eqref{eq:def b}, \eqref{eq:def c}. Thus the partition $P$
satisfies~\eqref{eq:diam bound union}, completing the proof of
Lemma~\ref{lem:union frag}.
\end{proof}


\begin{proof}[Proof of Theorem~\ref{thm:UMunion}] Assume first that
$U_1,U_2$ are bounded. Define a sequence $\{P_k\}_{k=0}^\infty$ of
partitions of $U_1\cup U_2$ as follows. Start with the trivial
partition $P_0=\{U_1\cup U_2\}$, and having defined $P_k$, the
partition $P_{k+1}$ is obtained by applying Lemma~\ref{lem:union
frag} to the sets $U_1\cap C$ and $U_2\cap C$ for each $C\in P_k$.
Then the partitions $\{P_k\}_{k=0}^\infty$ have the following
properties.
\begin{itemize}
\item $P_{k+1}$ is a refinement of $P_k$,
\item for every $C\in P_k$ we have
\begin{equation}\label{eq:diam vanishes}
\diam_d(C)\le (1-\delta)^k\diam_d(U_1\cup U_2), \end{equation}
\item for every distinct $C_1,C_2\in P_{k+1}$ such that $C_1,C_2\subseteq
C$ for some $C\in P_k$, we have
\begin{equation}\label{eq:iterated sep}
d(C_1,C_2)\ge \frac{\diam_d(C)}{D_1D_2+2D_1+2D_2+2+\e}.
\end{equation}
\end{itemize}

It follows from~\eqref{eq:diam vanishes} that for every distinct
$x,y\in U_1\cup U_2$ we have $P_k(x)\neq P_k(y)$ for $k\ge 0$
large enough. Thus for distinct $x,y\in U_1\cup U_2$ let $k(x,y)$
denote the largest integer $k\ge 0$ such that $P_k(x)=P_k(y)$.
Define
\begin{equation*}\label{eq:def rho}
\rho(x,y)=\left\{\begin{array}{ll}\diam_d\left(P_{k(x,y)}(x)\right)
& x\neq y,\\0 & x=y.\end{array}\right.
\end{equation*}
Then $\rho$ is an ultrametric on $U_1\cup U_2$. Indeed, for distinct
$x,y,z\in U_1\cup U_2$ let $k\ge 0$ be the largest integer such that
$P_k(x)=P_k(y)=P_k(z)$. Then $k=\min\{k(x,z),k(y,z)\}$ and
$P_k(x)\supseteq P_{k(x,y)}(x)$, implying that
$
\rho(x,y)=\diam_d\left(P_{k(x,y)}(x)\right)\le
\diam_d\left(P_k(x)\right)=\max\{\rho(x,z),\rho(y,z)\}.
$
For distinct $x,y\in U_1\cup U_2$, since $x,y\in P_{k(x,y)}(x)$, we have
$
\rho(x,y)=\diam_d\left(P_{k(x,y)}(x)\right)\ge d(x,y),
$
while since $P_{k(x,y)+1}(x)\neq P_{k(x,y)+1}(y)$, we deduce
from~\eqref{eq:iterated sep} that
\begin{multline*}
d(x,y)\ge d\left(P_{k(x,y)+1}(x),P_{k(x,y)+1}(y)\right)\\\ge
\frac{\diam_d\left(P_{k(x,y)}(x)\right)}{D_1D_2+2D_1+2D_2+2+\e}=\frac{\rho(x,y)}{D_1D_2+2D_1+2D_2+2+\e}.
\end{multline*}
The above argument shows that if $U_1,U_2$
are bounded, the metric space $(U_1\cup U_2,d)$ embeds with distortion $D_1D_2+2D_1+2D_2+2+\e$ into an ultrametric space for every $\e\in(0,1)$.

For possibly unbounded $U_1,U_2\subseteq X$, fix $x_0\in X$, and for
every $n\in \N$ let $\rho_n$ be an ultrametric on $\left(U_1\cap
B_d(x_0,n)\right)\cup\left(U_2\cap B_d(x_0,n)\right)$ satisfying
$$d(x,y)\le \rho_n(x,y)\le \left(D_1D_2+2D_1+2D_2+2+\frac{1}{n}\right)d(x,y)$$ for all
$x,y\in \left(U_1\cap B_d(x_0,n)\right)\cup\left(U_2\cap
B_d(x_0,n)\right)$. Define also $\rho_n(x,y)=0$ if $\{x,y\}$ is not
contained in $\left(U_1\cap B_d(x_0,n)\right)\cup\left(U_2\cap
B_d(x_0,n)\right)$. Let $\U$ be a free ultrafilter on $\N$ and set
$$
\rho_\infty(x,y)\eqdef \lim_{n\to \U} \rho_n(x,y).
$$
Then $\rho_\infty$ is an ultrametric on $U_1\cup U_2$ satisfying
$d\le \rho_\infty\le (D_1D_2+2D_1+2D_2+2)d$.
\end{proof}

\begin{remark}\label{rem:euclidean question}
There are several interesting variants of the problem studied in Theorem~\ref{thm:UMunion}. For example, answering our initial question, Konstantin Makarychev and Yury Makarychev proved (private communication) that if $(X,d)$ is a metric space and $E_1,E_2\subseteq X$ embed into Hilbert space with distortion $D\ge 1$ then $(E_1\cup E_2,d)$ embeds into Hilbert space with distortion $f(D)$. Their argument crucially uses Hilbert space geometry, and therefore the following natural question remains open: if $(X,d)$ is a metric space and $E_1,E_2\subseteq X$ embed into $L_1$ with distortion $D\ge 1$, does it follow that $(E_1\cup E_2,d)$ embeds into $L_1$ with distortion $f(D)$? Regarding unions of more than two subsets, perhaps even the following (ambitious) question has a positive answer: if $E_1,\ldots,E_n\subseteq X$ embed into Hilbert space with distortion $D\ge 1$, does it follow that $(E_1\cup\ldots \cup E_n,d)$ embeds into Hilbert space with distortion $O(\log n)f(D)$? If true, this statement (in the isometric case $D=1$) would yield a very interesting strengthening of Bourgain's embedding theorem~\cite{Bourgain-embed}, which asserts that any $n$-point metric space embeds into Hilbert space with distortion $O(\log n)$.
\end{remark}

\begin{remark}\label{rem:D1D2 sharp}
The following example shows that Theorem~\ref{thm:UMunion} is sharp up to lower order terms. Fix two integers $ M,N\ge 2$, and write $MN=K(M+N)+L$, where $K\in \N$ and $L\in \{0,1,\ldots,M+N-1\}$. Consider the following two subsets of the real line:
\begin{align*}
U_1&\eqdef\bigcup_{i=0}^{K-1}\left\{i(M+N),i(M+N)+1,\ldots, i(M+N)+M-1\right\}, \\
U_2&\eqdef\bigcup_{i=0}^{K-1} \left\{i(M+N)+M, i(M+N)+M+1,\ldots, (i+1)(M+N)-1\right\} .
\end{align*}
For $i\in \{1,2\}$, let $D_i\ge 1$ be the best possible distortion of $U_i$ (with the metric inherited from $\mathbb R$) in an ultrametric space. For $i_1,i_2\in \{0,\ldots,K-1\}$ and $j_1,j_2\in \{0,\ldots,M-1\}$ define
\begin{equation*}
\rho_1\left(i_1(M+N)+j_1,i_2(M+N)+j_2\right)\eqdef \left\{\begin{array}{ll}0 & i_1=i_2\ \wedge\ j_1=j_2,\\ M-1 & i_1=i_2\ \wedge\ j_1\neq j_2,\\
(K-1)(M+N)+M-1&i_1\neq i_2.\end{array}\right.
\end{equation*}
Then $\rho_1$ is an ultrametric on $U_1$ satisfying $|x-y|\le \rho_1(x,y)\le (M-1)|x-y|$
for all $x,y\in U_1$. Hence $D_1\le M-1$.
Similarly, for $i_1,i_2\in \{0,\ldots,K-1\}$ and $j_1,j_2\in \{M,\ldots,M+N-1\}$ define
\begin{equation*}
\rho_2\left(i_1(M+N)+j_1,i_2(M+N)+j_2\right)\eqdef \left\{\begin{array}{ll}0 & i_1=i_2\ \wedge\ j_1=j_2,\\ N-1 & i_1=i_2\ \wedge\ j_1\neq j_2,\\
(K-1)(M+N)+N-1&i_1\neq i_2.\end{array}\right.
\end{equation*}
Then $\rho_2$ is an ultrametric on $U_2$ satisfying $|x-y|\le \rho_1(x,y)\le (N-1)|x-y|$
for all $x,y\in U_2$. Hence $D_2\le N-1$.
But $U_1\cup U_2=\{0,1,\ldots, K(M+N)-1\}$, and hence any embedding of $U_1\cup U_2$ into an ultrametric space incurs distortion at least $K(M+N)-1$ (see for example~\cite[Lem.~2.4]{MN04}). Observe that this lower bound on the distortion equals $MN-L-1\ge (M-1)(N-1)-1\ge D_1D_2-1$.
When $L=0$ (e.g., when $M=N=2S$ or $M=2N=6S$ for some $S\in\mathbb N$), the above distortion lower bound becomes
$MN-L-1\ge (D_1+1)(D_2+1)-1$. Thus one cannot improve
the bound in Theorem~\ref{thm:UMunion} to $D_1D_2$, i.e.,
additional lower order terms are necessary.
\end{remark}

%

\subsection*{Acknowledgements.} We are grateful to the following people for helpful  suggestions: Keith Ball, Witold Bednorz, Rafa{\l} Lata{\l}a, Gilles Pisier, Gideon Schechtman, Michel Talagrand.



\bibliographystyle{abbrv}
\bibliography{hausdorff}

\end{document}